\documentclass[11pt]{article}
\usepackage[utf8]{inputenc}
\usepackage[OT1,T1]{fontenc}
\usepackage{amsmath,amssymb}
\usepackage{mathrsfs}
\usepackage[frenchb]{babel}
\usepackage[colorlinks,citecolor=black,linkcolor=black,urlcolor=black,
  pdfpagemode=None,pdfstartview=FitH]{hyperref}

\newtheorem{theorem}{Théorème}
\newtheorem{lemma}[theorem]{Lemme}
\newenvironment{proof}{\trivlist
  \item[\hskip\labelsep{\itshape Preuve.}]\upshape}{\nobreak\noindent
  $\square$\endtrivlist}

\DeclareMathOperator\Hom{Hom}
\DeclareMathOperator\Ext{Ext}
\DeclareMathOperator\Tor{Tor}
\DeclareMathOperator\hd{hd}
\DeclareMathOperator\soc{soc}
\DeclareMathOperator\wt{wt}

\begin{document}
\title{Réflexions dans un cristal}
\author{Pierre Baumann, Stéphane Gaussent et Joel Kamnitzer}
\date{}
\maketitle

\begin{abstract}
\noindent
Soit $\mathfrak g=\mathfrak n^-\oplus\mathfrak h\oplus\mathfrak n^+$
une algèbre de Kac-Moody symétrisable. Soit $B(\infty)$ le cristal
de Kashiwara de $U_q(\mathfrak n^-)$, soit $\lambda$ un poids
dominant, soit $T_\lambda=\{t_\lambda\}$ le cristal à un élément
de poids $\lambda$, et soit $B(\lambda)\subseteq B(\infty)\otimes
T_\lambda$ le cristal de la représentation intégrable de plus haut
poids $\lambda$. Nous calculons les paramètres en cordes descendants
d'un élément $b\otimes t_\lambda$ de $B(\lambda)$ en fonction des
paramètres de Lusztig de $b$.
\end{abstract}

\section{Énoncé du résultat}
Soit $\mathfrak g$ une algèbre de Kac-Moody symétrisable.
Elle vient avec la décomposition triangulaire
$\mathfrak g=\mathfrak n^-\oplus\mathfrak h\oplus\mathfrak n^+$,
la famille $(\alpha_1,...,\alpha_n)$ des racines simples,
la famille $(\alpha_1^\vee,...,\alpha_n^\vee)$ des coracines
simples, et le groupe de Weyl $W$ engendré par les réflexions
simples $s_1$, ..., $s_n$. Soit $P^+$ l'ensemble des poids entiers
dominants.

À chaque $\lambda\in P^+$ correspond une représentation
irréductible intégrable $L(\lambda)$ de $\mathfrak g$. Le choix d'un
vecteur de plus haut poids $v_\lambda$ donne lieu à une surjection
$X\mapsto X\cdot v_\lambda$ de $U(\mathfrak n^-)$ sur $L(\lambda)$.
La combinatoire de cette situation est régie par les cristaux de
Kashiwara~\cite{Kashiwara95} : le cristal $B(\lambda)$ de $L(\lambda)$
est un sous-cristal de $B(\infty)\otimes T_\lambda$, où $B(\infty)$
est le cristal de $U_q(\mathfrak n^-)$ et $T_\lambda=\{t_\lambda\}$
est le cristal à un élément de poids $\lambda$. Le cristal
$B(\lambda)$ étant normal, nous avons
$$\varphi_i(b\otimes t_\lambda)=\max\{n\in\mathbb N\mid\tilde
f_i^n(b\otimes t_\lambda)\in B(\lambda)\}$$
pour tout $i\in\{1,...,n\}$ et tout $b\otimes t_\lambda\in B(\lambda)$.

Le cristal $B(\infty)$ est muni d'une involution $*$ (voir
\cite{Kashiwara95}, \S8.3), ce qui permet de définir des opérateurs
$\tilde e_i^*=*\tilde e_i*$. Dans \cite{Saito94}, Saito définit une
bijection
$$\sigma_i:\{b\in B(\infty)\mid\varepsilon_i(b)=0\}\to
\{b\in B(\infty)\mid\varepsilon_i(b^*)=0\}$$
par $\sigma_i(b)=\tilde f_i^{\varphi_i(b^*)}\tilde e_i^{*\max}b$.
Prolongeons $\sigma_i$ à $B(\infty)$ en posant
$\hat\sigma_i(b)=\sigma_i(\tilde e_i^{\max}b)$ pour tout $b\in B(\infty)$.

\begin{theorem}
\label{th:RefCris}
Soit $\lambda\in P^+$, soit $(s_{i_1},\ldots,s_{i_\ell})$ une
décomposition réduite dans $W$, et soit $b\in B(\infty)$ tel que
$b\otimes t_\lambda\in B(\lambda)$. Définissons par récurrence
des éléments $b_0$, ..., $b_\ell$ de $B(\infty)$ et des entiers
$c_1$, ..., $c_\ell$~par
$$b_0=b,\quad c_k=\varphi_{i_k}(b_{k-1}\otimes t_\lambda),\quad
b_k\otimes t_\lambda=\tilde f_{i_k}^{c_k}(b_{k-1}\otimes t_\lambda),$$
pour tout $k\in\{1,\ldots,\ell\}$. Posons enfin
$d_k=\langle\alpha_{i_k}^\vee,s_{i_{k-1}}\cdots
s_{i_1}\lambda\rangle$. Alors
\begin{equation}
\label{eq:RefCris}
\hat\sigma_{i_\ell}\cdots\hat\sigma_{i_1}b=\tilde
e_{i_\ell}^{*d_\ell}\cdots\tilde e_{i_1}^{*d_1}b_\ell.
\end{equation}
\end{theorem}

Le cas particulier $\ell=1$ de ce résultat est dû à Muthiah et
Tingley (\cite{MuthiahTingley12}, proposition~2.2).

Plaçons-nous sous les hypothèses du théorème~\ref{th:RefCris}.
Prenant $k\in\{1,\ldots,\ell\}$, écrivant~\eqref{eq:RefCris} pour
la décomposition réduite $(s_{i_1},\ldots,s_{i_k})$, et égalant
les poids des deux membres, nous trouvons
\begin{equation}
\label{eq:SommetPolMV}
s_{i_1}\cdots s_{i_k}\wt(\hat\sigma_{i_k}\cdots\hat
\sigma_{i_1}b)+\lambda=s_{i_1}\cdots s_{i_k}\wt(b_k\otimes t_\lambda).
\end{equation}
Notant $n_1$, ..., $n_\ell$ les entiers solutions du système
d'équations
$$s_{i_1}\cdots s_{i_k}\wt(\hat\sigma_{i_k}\cdots\hat\sigma_{i_1}b)-
\wt(b)=\sum_{p=1}^kn_p\;s_{i_1}\cdots s_{i_{p-1}}\alpha_{i_p},$$
pour $k\in\{1,\ldots,\ell\}$, nous obtenons alors
\begin{equation}
\label{eq:FormSophie}
n_k=-c_k-\langle\alpha_{i_k}^\vee,\wt(b_k\otimes t_\lambda)\rangle.
\end{equation}

Supposons qu'en outre $\mathfrak g$ soit de dimension finie, et
prenons pour $(s_{i_1},\ldots,s_{i_\ell})$ une décomposition réduite
de l'élement le plus long de~$W$. Il suit alors de \cite{Saito94}
que les entiers $n_1$, ..., $n_\ell$ sont les paramètres de Lusztig
de $b$ relativement à cette décomposition réduite (voir par
exemple \cite{BerensteinZelevinsky01}, \S3 pour cette notion).
Dans ces conditions, la formule~\eqref{eq:FormSophie} est équivalente
à la relation découverte par Morier-Genoud~\cite{Morier-Genoud03} entre
les paramètres de Lusztig de $b$ et les paramètres en cordes de $b'$,
où $b'\otimes t_\lambda$ est l'image de $b\otimes t_\lambda$ par
l'involution de Schützenberger de~$B(\lambda)$.

Toujours dans le cas où $\mathfrak g$ est de dimension finie,
l'égalité \eqref{eq:SommetPolMV} reflète la possibilité d'exprimer
les sommets du polytope de Mirković-Vilonen de $b\otimes t_\lambda$
de deux façons différentes : soit en termes de réflexions de Saito,
autrement dit en termes de paramètres de Lusztig, soit en termes
de la structure de cristal de $B(\lambda)$ (voir \cite{Ehrig10}, \S6).
La principale motivation du présent travail réside dans la possibilité
d'étendre partiellement ce résultat au cas où~$\mathfrak g$ est une
algèbre de Kac-Moody symétrisable.

\section{Démonstration}
Ramenons-nous au cas d'une matrice de Cartan symétrique grâce à
la méthode de repliement du diagramme de Dynkin (voir par exemple
\cite{Kashiwara96}, \S5). Soit $\Lambda$ la $\mathbb C$-algèbre
préprojective complétée construite sur le graphe de Dynkin de
$\mathfrak g$. En vecteur-dimension $\nu$, les structures de
$\Lambda$-module sont les points d'une variété affine $\Lambda(\nu)$,
appelée variété nilpotente de Lusztig. Les composantes irréductibles
de ces variétés sont indexées par $B(\infty)$ (voir
\cite{KashiwaraSaito97}, théorème~5.3.1) : à un élément
$b\in B(\infty)$ de poids $-\nu$ est associée une composante
irréductible $\Lambda_b$ de $\Lambda(\nu)$. Cette bijection permet
de lire les opérations de cristal de $B(\infty)$ en termes
d'opérations algébriques sur les $\Lambda$-modules. Le
théorème~\ref{th:RefCris} se trouve alors être la traduction
du théorème~\ref{th:DecMod} ci-dessous.

Pour $i\in\{1,...,n\}$, notons $S_i$ le $\Lambda$-module simple de
vecteur-dimension $\alpha_i$ et $I_i$ l'annulateur de $S_i$.
Nous définissons le $i$-socle $\soc_iM$ (respectivement, la
$i$-tête $\hd_iM$) d'un $\Lambda$-module $M$ comme étant le
plus grand sous-module (respectivement, quotient) de $M$
isomorphe à une somme directe de copies de $S_i$.
Alors $\soc_iM\cong\Hom_\Lambda(\Lambda/I_i,M)$
et $\hd_iM\cong(\Lambda/I_i)\otimes_\Lambda M$.

Si $(s_{i_1},\ldots,s_{i_\ell})$ est une décomposition réduite dans
$W$, alors le $\Lambda$-bimodule
$I_{i_1}\otimes_\Lambda\cdots\otimes_\Lambda I_{i_\ell}$ est
isomorphe à l'idéal produit $I_{i_1}\cdots I_{i_\ell}$ (voir
\cite{BuanIyamaReitenScott09}). Ce dernier ne dépend que du produit
$w=s_{i_1}\cdots s_{i_\ell}$ et peut donc être désigné par la
notation~$I_w$. Comme $I_w$ est basculant, la théorie de Brenner-Butler
fournit deux paires de torsion $(\mathscr T_w,\mathscr F_w)$ et
$(\mathscr T^w,\mathscr F^w)$ dans la catégorie des $\Lambda$-modules
de dimension finie, données par
\begin{xalignat*}2
\mathscr T_w&=\{T\mid I_w\otimes_\Lambda T=0\},&
\mathscr F_w&=\{T\mid\Tor_1^\Lambda(I_w,T)=0\},\\
\mathscr T^w&=\{T\mid\Ext^1_\Lambda(I_w,T)=0\},&
\mathscr F^w&=\{T\mid\Hom_\Lambda(I_w,T)=0\}.
\end{xalignat*}
Le sous-module de torsion d'un $\Lambda$-module $M$ de dimension
finie relativement à $(\mathscr T_w,\mathscr F_w)$ (respectivement,
$(\mathscr T^w,\mathscr F^w)$) est noté $M_w$ (respectivement, $M^w$).

\begin{lemma}
\label{le:LemmeIter}
Soit $w\in W$, soit $i\in\{1,...,n\}$, et soit $M$ un $\Lambda$-module
de dimension finie. Supposons que $s_iw>w$ et que
$\Ext^1_\Lambda(S_i,M)=0$. Alors $M^{s_iw}\cong I_i\otimes_\Lambda M^w$.
\end{lemma}
\begin{proof}
Pour commencer, observons que $\soc_iM\in\mathscr T_{s_i}$
(\cite{BaumannKamnitzerTingley11}, exemple~5.6~(i)), d'où
$I_i\otimes_\Lambda\soc_iM=0$. Posons $N=\Hom_\Lambda(I_i,M)$.
Appliquant deux fois le théorème~5.4~(i) de
\cite{BaumannKamnitzerTingley11}, nous obtenons
$$M^{s_iw}\cong I_{s_iw}\otimes_\Lambda\Hom_\Lambda
(I_{s_iw},M)\cong I_i\otimes_\Lambda
I_w\otimes_\Lambda\Hom_\Lambda(I_w,N)\cong I_i\otimes_\Lambda N^w.$$

Compte tenu de l'isomorphisme $S_i\cong\Lambda/I_i$,
l'hypothèse $\Ext^1_\Lambda(S_i,M)=0$ conduit à la suite exacte
$0\to\soc_iM\to M\to N\to0$. L'hypothèse $s_iw>w$ entraîne que
$\mathscr T_{s_i}\subseteq\mathscr T^w$
(\cite{BaumannKamnitzerTingley11}, proposition~5.16), d'où
$\soc_iM\in\mathscr T^w$. Nous avons ainsi une suite exacte
$0\to\soc_iM\to M^w\to N^w\to0$, d'où nous déduisons que
$I_i\otimes_\Lambda M^w\cong I_i\otimes_\Lambda N^w\cong M^{s_iw}$.
\end{proof}

Pour un $\Lambda$-module $M$ et un mot $(i_1,\ldots,i_\ell)$,
définissons les sous-modules $\soc_{(i_1,\ldots,i_k)}M$ de $M$ par
récurrence sur $k\in\{0,\ldots,\ell\}$ de la façon suivante
(\cite{GeissLeclercSchroer11}, \S2.4) :
$$\soc_{()}M=0,\quad\soc_{(i_1,\ldots,i_k)}M/\soc_{(i_1,\ldots,
i_{k-1})}M=\soc_{i_k}(M/\soc_{(i_1,\ldots,i_{k-1})}M).$$
Pour $i\in\{1,...,n\}$, notons $\hat I_i$ l'enveloppe injective de $S_i$.
Pour $\lambda\in P^+$, posons
$\hat I_\lambda=\bigoplus_{i=1}^n\hat I_i^{\oplus\lambda_i}$, où
$\lambda_i=\langle\alpha_i^\vee,\lambda\rangle$.

\begin{lemma}
\label{le:LemmeGLS}
Soit $\lambda\in P^+$, soit $w\in W$, et soit $(s_{i_1},\ldots,s_{i_\ell})$
une décomposition réduite de $w$. Alors $\soc_{(i_1,\ldots,i_\ell)}\hat
I_\lambda$ est le plus grand sous-module de $\hat I_\lambda$ appartenant à
$\mathscr T_w$ et son vecteur-dimension est égal à $\lambda-w^{-1}\lambda$.
\end{lemma}
\begin{proof}
L'énoncé équivaut à dire que
$\soc_{(i_1,\ldots,i_\ell)}\hat I_\lambda\in\mathscr T_w$ et que
$\Hom_\Lambda(X,\hat I_\lambda/\soc_{(i_1,\ldots,i_\ell)}\hat I_\lambda)=0$
pour tout module $X\in\mathscr T_w$. Par additivité, nous pouvons donc
nous ramener au cas où $\lambda$ est un poids fondamental, c'est-à-dire
$\hat I_\lambda$~est un module indécomposable~$\hat I_j$.

Le module $\soc_{(i_1,\ldots,i_\ell)}\hat I_j$ est le module noté
$I_{\mathbf i,j}$ dans \cite{GeissLeclercSchroer11}, \S2.4, avec
$\mathbf i=(i_1,\ldots,i_\ell)$. C'est un objet injectif de la
catégorie $\mathscr T_w$ (\cite{GeissLeclercSchroer11},
théorème~2.8~(iii) et \cite{BaumannKamnitzerTingley11}, exemple~5.15).
Montrons qu'il contient tous les sous-modules de $\hat I_j$
appartenant à $\mathscr T_w$. Soit $X$ un tel sous-module.
Alors la somme $Y=X+\soc_{(i_1,\ldots,i_\ell)}\hat I_j$ appartient à
$\mathscr T_w$. Comme $\soc_{(i_1,\ldots,i_\ell)}\hat I_j$ est injectif
dans $\mathscr T_w$, il est un facteur direct de $Y$. Or tous les
sous-modules non-nuls de $\hat I_j$ contiennent son socle $S_j$,
ce qui exclut l'existence de somme directe non-triviale à l'intérieur
de $\hat I_j$. Par conséquent, $\soc_{(i_1,\ldots,i_\ell)}\hat I_j$
est soit égal à $Y$ tout entier, soit réduit à~$0$. La seconde
possibilité a lieu quand $j\notin\{i_1,\ldots,i_\ell\}$, mais dans
ce cas $X$ est lui aussi nul, car pour des raisons de
vecteur-dimension, il ne peut pas contenir le socle $S_j$ de
$\hat I_j$ (\cite{GeissLeclercSchroer11}, corollaire~9.3 et
lemme~10.2). Donc $\soc_{(i_1,\ldots,i_\ell)}\hat I_j$ contient $X$
dans les deux cas de figure.

Enfin, l'assertion sur le vecteur-dimension est prouvée dans
\cite{GeissLeclercSchroer11}, corollaire~9.2.
\end{proof}

\begin{theorem}
\label{th:DecMod}
Soit $\lambda\in P^+$, soit $w\in W$, soit $(s_{i_1},\ldots,s_{i_\ell})$
une décomposition réduite de $w$, et soit $M$ un sous-module
de dimension finie de $\hat I_\lambda$. Construisons par récurrence
une chaîne $M_0\subseteq M_1\subseteq\cdots\subseteq M_\ell$ de
sous-modules de $\Hat I_\lambda$ de la façon suivante :
$$M_0=M,\quad M_k/M_{k-1}=\soc_{i_k}(\hat I_\lambda/M_{k-1}).$$
Alors $(M_\ell)_w=\soc_{(i_1,\ldots,i_\ell)}\hat I_\lambda$ et
$M_\ell/(M_\ell)_w\cong\Hom_\Lambda(I_w,M)$.
\end{theorem}
\begin{proof}
Des arguments classiques montrent que si $f:X\to Y$ est un
homomorphisme de $\Lambda$-modules, alors
$f(\soc_{(i_1,\ldots,i_\ell)}X)\subseteq\soc_{(i_1,\ldots,i_\ell)}Y$.
Appliquant ce résultat à la surjection $\hat I_\lambda\to\hat I_\lambda/M$,
nous obtenons que $\soc_{(i_1,\ldots,i_\ell)}\hat I_\lambda$ est
inclus dans $M_\ell$. La première égalité de l'énoncé découle alors
du lemme~\ref{le:LemmeGLS}.

Par ailleurs, $S_{i_k}\in\mathscr F^{s_{i_k}\cdots s_{i_\ell}}$
pour tout $k\in\{1,\ldots,\ell\}$, d'où
$(M_{k-1})^{s_{i_k}\cdots s_{i_\ell}}=(M_k)^{s_{i_k}\cdots s_{i_\ell}}$.
La définition de $M_k$ entraîne que $\soc_{i_k}(\hat I_\lambda/M_k)=0$,
autrement dit $\Hom_\Lambda(S_{i_k},\hat I_\lambda/M_k)=0$ ; compte
tenu du caractère injectif de $\hat I_\lambda$, cela donne
$\Ext^1_\Lambda(S_{i_k},M_k)=0$. Une utilisation répétée du
lemme~\ref{le:LemmeIter} conduit alors à
$M^w\cong I_w\otimes_\Lambda M_\ell$. Utilisant les relations
$M_\ell/(M_\ell)_w\cong\Hom_\Lambda(I_w,I_w\otimes_\Lambda M_\ell)$
(\cite{BaumannKamnitzerTingley11}, théorème~5.4~(ii))
et $\Hom_\Lambda(I_w,M/M^w)=0$, nous obtenons la seconde égalité
de l'énoncé.
\end{proof}

Nous pouvons maintenant démontrer le théorème~\ref{th:RefCris}.
Adoptons les notations utilisées dans son énoncé. Soit $M$ un point
général de la composante irréductible $\Lambda_b$. La condition
$b\otimes t_\lambda\in B(\lambda)$ se traduit par
$\varepsilon_i(b^*)\leq\langle\alpha_i^\vee,\lambda\rangle$
pour tout $i\in\{1,...,n\}$ (\cite{Kashiwara95}, proposition~8.2),
d'où $\dim\soc_iM\leq\dim\soc_i\hat I_\lambda$. Il existe
donc une inclusion $M\hookrightarrow\hat I_\lambda$. Les modules
$M_k$ construits dans l'énoncé du théorème~\ref{th:DecMod}
sont alors des points généraux des composantes $\Lambda_{b_k}$ et
le module $\Hom_\Lambda(I_w,M)$ est un point général de
$\Lambda_{\hat\sigma_{i_\ell}\cdots\hat\sigma_{i_1}b}$
(\cite{BaumannKamnitzerTingley11}, proposition~5.24).
Il reste à observer que $\soc_{(i_1,\ldots,i_k)}\hat
I_\lambda/\soc_{(i_1,\ldots, i_{k-1})}\hat I_\lambda$
est isomorphe à la somme directe de $d_k$ copies de~$S_{i_k}$.

Pierre Baumann,
Institut de Recherche Mathématique Avancée,
Université de Strasbourg et CNRS,
7 rue René Descartes,
67084 Strasbourg Cedex,
France.\\
\texttt{p.baumann@unistra.fr}
\medskip

Stéphane Gaussent,
Université de Lyon,
ICJ (UMR 5208),
Université Jean Monnet,
42023 Saint-Etienne Cedex 2,
France.\\
\texttt{stephane.gaussent@univ-st-etienne.fr}
\medskip

Joel Kamnitzer,
Department of Mathematics,
University of Toronto,
Toronto, ON, M5S 2E4
Canada.\\
\texttt{jkamnitz@math.toronto.edu}
\end{document}